\documentclass{amsart}

\usepackage{xcolor}
\usepackage{mathtools}
\usepackage{amsmath}
\usepackage{amsthm}
\usepackage{graphicx}
\usepackage{amssymb}
\usepackage{epstopdf}
\usepackage{nicefrac}
\usepackage{mathrsfs}
\usepackage{hyperref}
\usepackage{enumitem}
\usepackage{tikz}
\usepackage{caption}
\usepackage{subcaption}
\usepackage{cite}
\usetikzlibrary{calc,arrows.meta,decorations.pathreplacing}

\theoremstyle{plain}
\newtheorem{theo}{Theorem}[section]
\newtheorem{lm}[theo]{Lemma}
\newtheorem{cor}[theo]{Corollary}
\newtheorem{prop}[theo]{Proposition}

\newtheorem*{ack}{Acknowledgements}

\newtheorem*{example*}{Example}

\DeclareMathOperator{\conv}{conv}

\DeclareMathOperator{\Ha}{H}

\DeclareMathOperator{\PW}{PW}
\DeclareMathOperator{\supp}{supp}

\DeclareMathOperator{\Image}{Im}

\title[Helson inequality and Hankel operators for Paley-Wiener spaces]{Helson Inequality, Hankel Operators, and Weak Factorization on Paley-Wiener spaces of Convex Domains}
\author[K. Bampouras]{Konstantinos Bampouras}
\address{Department of Mathematics and Systems Analysis, School of Science, Aalto University, Espoo, Finland}
\email{konstantinos.bampouras@aalto.fi}

\date{\today}

\begin{document}
	
	\begin{abstract}
	For any convex set $\Omega\subset\mathbb{R}^n$ that does not contain affine lines, we prove the inequality
	$$\int_\Omega\frac{|\hat{f}(x)|^2}{\omega_\Omega(x)}\,dx\leq C(n)\|f\|_{L^1}^2,\quad \supp\hat{f}\subset\Omega,$$
	where $\omega_\Omega(x)=m(\Omega\cap(2x-\Omega))$. As a consequence, we derive a weak factorization for
	$$\PW^1(\Omega)=\{f\in L^1(\mathbb{R}^n):\supp\hat{f}\subset\Omega\}.$$
	Furthermore, we establish a complete characterization of Schatten class Hankel operators for polyhedra for all $1\leq p<\infty,$ extending the already known $1\leq p\leq 2$  range.
	
	\end{abstract}
	
	\maketitle

\section{Introduction}
Let $\Omega$ be an open subset of $\mathbb{R}^n.$ The Paley-Wiener spaces associated with $\Omega$ are defined as
$$\PW^p(\Omega)=\{f\in L^p(\mathbb{R}^n):\supp\hat{f}\subset\Omega\}, \quad 1\leq p\leq \infty,$$
endowed with the $L^p$ norm. When $p=2$ we may also write $\PW(\Omega):=\PW^2(\Omega)$. The case $\Omega=\mathbb{R}^n_+$ coincides with the classical multivariable Hardy spaces $\PW^p(\mathbb{R}^n_+)\cong H^p(\mathbb{C}_+^n)$. For Hardy spaces, H. Helson \cite{MR2263964} proved that
$$\int_{\mathbb{R}^n_+}\dfrac{|\hat{f}(x_1,\dots,x_n)|^2}{x_1\dots x_n}dx_1\dots dx_n\leq \|f\|_{H^1}^2,$$
which is the $n$-dimensional extension of the classical Carleman inequality.

In order to generalise this inequality to broader contexts, one needs to identify the correct analogue of the weight $x_1\dots x_n$. It turns out that, for convex sets, the appropriate function is the Hilbert-Schmidt weight for Hankel operators (see Section \ref{sec3}),
$$\omega_\Omega(x)=m(\Omega\cap (2x-\Omega)).$$
Notice that this function is finite if and only if $\Omega$ does not contain affine lines. Thus, for such a convex set $\Omega$, the natural analogue of Helson's inequality takes the form
$$\int_{\Omega}\dfrac{|\hat{f}(x)|^2}{\omega_\Omega(x)}dx\leq C(\Omega) \|f\|_{L^1}^2,\quad f\in\PW^1(\Omega).$$
In the recent paper \cite{bampouras2026boundaryweightedfourierinequalitiesconvex}, the authors proved Helson's inequality for all polyhedra $P$, with a constant $C(P)$ depending on the dimension and the number of facets of $P$. Our first result shows that Helson's inequality in fact holds for every convex set with no affine lines, with a uniform constant depending only on the dimension.

\begin{theo}\label{helson}
	Let $\Omega\subset \mathbb{R}^n$ be a convex set that does not contain affine lines. Then it holds that
	$$\int_{\Omega}\dfrac{|\hat{f}(x)|^2}{\omega_\Omega(x)}dx\leq \frac{2^n(2n)!\Gamma(n/2+1)^2}{(n!)^3}\|f\|_{L^1}^2, \quad f\in\PW^1(\Omega).$$
\end{theo}

The appearance of the function $\omega_\Omega$ is closely connected to the theory of Hankel operators. For a distribution $\widehat{\phi}$, the Hankel operator $\Ha_\phi:\PW(\Omega)\to \PW(\Omega)$ is defined by
$$\widehat{\Ha_\phi f}(x)=\int_\Omega \widehat{\phi}(x+y)\hat{f}(y)dy, \quad x\in \Omega.$$
It is evident that $\|\Ha_\phi\|\leq \|\phi\|_{L^\infty}$. Z. Nehari \cite{nehari1957bounded} showed that for any bounded Hankel operator on $\PW(\mathbb{R}_+)$ there exists a bounded function $\psi$ such that $\Ha_\phi=\Ha_\psi$. The analogous result for the multivariable Hardy spaces is one of the most famous open problems in complex analysis. Nevertheless, H. Helson, using his inequality, was able to prove that every Hilbert-Schmidt Hankel operator on the multivariable Hardy spaces is generated by a bounded function. More recently, it has been proved that every Schatten class Hankel operator on $\PW(\mathbb{R}^2_+)$ is generated by a bounded symbol \cite{bampouras2026recoveringproductbmoschatten}. As a consequence of Theorem \ref{helson}, we obtain the corresponding Hilbert-Schmidt result for every convex set with no affine lines.

\begin{theo}\label{theo1.2}
	Let $\Omega\subset\mathbb{R}^n$ be a convex set that does not contain affine lines. Then every Hilbert-Schmidt Hankel operator on $\PW(\Omega)$ is generated by a bounded symbol.
\end{theo}

There is also a factorization interpretation of this result. In Section \ref{sec5} we show that Theorem \ref{theo1.2} is equivalent to the following weak factorization of $\PW^1(\Omega)$.

\begin{cor}
	Let $\Omega$ be a convex set that does not contain affine lines. Then there is a $C>0$ such that for any $f\in \PW^1(\Omega)$, there are $f_j,g_j$ orthogonal sequences in $\PW(\frac{1}{2}\Omega)$ with
	$$f=\lim_{N\to \infty} \sum_{j=1}^N f_j g_j \text{ in $\mathcal{W}_2(1/2 \; \Omega)$,} \quad  \text{and}\quad  \sum_j (\|f_j\|_{L^2}\|g_j\|_{L^2})^2\leq C\|f\|_{L^1}^2.$$
\end{cor}

Our final result concerns the characterization of Schatten class Hankel operators. In one dimension, V. Peller and R. Rochberg \cite{MR878246,Peller1988WienerHopf} showed that Schatten class Hankel operators can be characterized in terms of Besov spaces. The corresponding multivariable result for Hardy spaces was proved in \cite{PottSmith2004}, while related characterizations for other Paley-Wiener spaces, including those associated with the cube and the ball, were obtained in \cite{bampouras2024besovspacesschattenclass,Peng1989PaleyWiener,Peng1988Disk}. We complete this picture for polyhedra by proving that the Schatten class condition is characterized entirely by the corresponding Besov space.

\begin{theo} Let $1\leq p<\infty$. Then, for every polyhedron $P\subset \mathbb{R}^n$ that does not contain affine lines it holds that
	$$\Ha_\varphi\in S^p(\PW(P))
	\quad\Longleftrightarrow\quad
	\varphi\in B_{p,p}^{1/p}(2P),$$
	with equivalence of norms.
\end{theo}

	\begin{ack}
		The author was supported by the Research Council of Finland, project no. 371637. They are a member of the Finnish Centre of Excellence in Randomness and Structures. Furthermore, the author gratefully acknowledges the assistance of ChatGPT (GPT-5.6 Sol, OpenAI), whose suggestions helped generate some of the ideas developed in this work and contributed to sharpening and refining several of the techniques and arguments, especially in Sections \ref{sec2} and \ref{Schatten}.

	\end{ack}

\section{Helson's Inequality}\label{sec2}
Let $\Omega\subset \mathbb{R}^n$ be a convex set that does not contain affine lines. For $p\geq 1,$ we define the associated Paley-Wiener spaces by
$$\PW^p(\Omega)=\{f\in L^p(\mathbb{R}^n):\supp\hat{f}\subset \Omega\},$$ endowed with the $L^p$ norm.

For a triple $(p,q,d),$ $1\leq p,q<\infty,$ $d\in\mathbb{R}$, we can consider the boundary-weighted Fourier inequality
$$\int_\Omega \dfrac{|\hat{f}(x)|^p}{\omega_\Omega^d(x)}dx\lesssim \|f\|_{L^q}^p,\quad f\in \PW^q(\Omega),$$ where $\omega_\Omega(x)=m(\Omega\cap (2x-\Omega))$, which was recently introduced and studied in \cite{bampouras2026boundaryweightedfourierinequalitiesconvex}. In this section we will be interested in the case $p=2,$ $q=d=1.$ We will refer to this inequality as Helson's inequality, since Helson proved the inequality for the classical Hardy spaces $\PW(\mathbb{R}^n_+)$.
This inequality has been proved for polyhedra \cite[Theorem 4.5]{bampouras2026boundaryweightedfourierinequalitiesconvex} and for the ball \cite[Theorem 5.3]{bampouras2026boundaryweightedfourierinequalitiesconvex}. The main theorem of this section, which is Theorem \ref{helson}, reads as follows.
\begin{theo}\label{helson1}
	Let $\Omega\subset \mathbb{R}^n$ be a convex set that does not contain affine lines. Then it holds that $$\int_{\Omega}\dfrac{|\hat{f}(x)|^2}{\omega_\Omega(x)}dx\leq \frac{2^n(2n)!\Gamma(n/2+1)^2}{(n!)^3}\|f\|_{L^1}^2, \quad f\in\PW^1(\Omega).$$
\end{theo}
First, let us note that it suffices to prove Theorem \ref{helson1} for bounded sets.
\begin{lm}\label{unbounded}
	It suffices to prove Theorem \ref{helson1} for bounded sets.
\end{lm}
\begin{proof}
	Let us assume that Theorem \ref{helson1} holds for bounded sets. By \cite[Lemma 2.6]{bampouras2026boundaryweightedfourierinequalitiesconvex}, functions in $\mathcal{F}^{-1}C_c^\infty(\Omega)$ are dense in $\PW^1(\Omega)$. Therefore, let us set $\Omega_r=\Omega\cap B(0,r)$. Since $\Omega_r\subset \Omega$, we have that $\omega_{\Omega_r}(x)\leq \omega_\Omega(x),$ $x\in\mathbb{R}^n$. Let $f\in \mathcal{F}^{-1}C_c^\infty(\Omega)$. Then there is $r>0$ with $\supp\hat{f}\subset \Omega_r$ and thus
	$$\int_\Omega \dfrac{|\hat{f}(x)|^2}{\omega_\Omega(x)}dx\leq\int_{\Omega_r} \dfrac{|\hat{f}(x)|^2}{\omega_{\Omega_r}(x)}dx\leq \frac{2^n(2n)!\Gamma(n/2+1)^2}{(n!)^3}\|f\|_{L^1}^2,$$ where the last inequality holds since $f\in \PW^1(\Omega_r)$.
\end{proof}
In order to state our next lemma, we need some notation. For a bounded convex set $\Omega$, its support function is defined by $$h_\Omega(x)=\sup_{y\in \Omega} \langle x,y\rangle.$$ The polar set of $\Omega$ is defined by
$$\Omega^\ast =\{x\in\mathbb{R}^n:\langle x,y\rangle\leq 1\text{ for all } y\in\Omega\}=h_\Omega^{-1}(-\infty,1].$$
The key function in this section will be the function
$$\Psi_\Omega(x)=\dfrac{(2\pi)^n}{n!}i^{-n}e^{2 \pi i h_\Omega(x)}.$$
\begin{lm}\label{psilemma}
	It holds that $$\widehat{\Psi_\Omega}(x)=m((\Omega-x)^\ast), \quad x\in\Omega,$$ where $\widehat{\Psi_\Omega}$ is understood as a distribution.
\end{lm}
\begin{proof}
	Let $$\Psi_{\Omega,\epsilon}(x)=\dfrac{(2\pi)^n}{n!}i^{-n}e^{2 \pi i h_\Omega(x)}e^{-2\pi \epsilon|x|}\in L^1(\mathbb{R}^n).$$ If $\sigma$ denotes the standard surface measure on $S^{n-1}$, using polar coordinates we can write
	\begin{eqnarray} \widehat{\Psi_{\Omega,\epsilon}}(x)&=&\int_{\mathbb{R}^n}\Psi_{\Omega,\epsilon}(y)e^{-2\pi i \langle x,y\rangle}dy=\dfrac{(2\pi)^n}{n!}i^{-n}\int_{\mathbb{R}^n}e^{2 \pi i h_\Omega(y)}e^{-2\pi \epsilon |y|}e^{-2\pi i \langle x,y\rangle}dy \nonumber \\
		&=&\dfrac{(2\pi)^n}{n!}i^{-n}\int_{S^{n-1}}\int_{0}^\infty r^{n-1}e^{2 \pi i r(h_\Omega(\theta)-\langle x,\theta\rangle+i\epsilon)}drd\sigma(\theta).  \nonumber 
	\end{eqnarray} 
	Since, via the Gamma function, $$\int_0^\infty r^{n-1}e^{2\pi i z r}dr=(n-1)!(2\pi)^{-n} i^n z^{-n}, \quad \Image z>0,$$ we get that
	
	$$ \widehat{\Psi_{\Omega,\epsilon}}(x)=\frac{1}{n}\int_{S^{n-1}}  (h_\Omega(\theta)-\langle x,\theta\rangle+i\epsilon)^{-n}d\sigma(\theta). $$
	Since $h_\Omega(\theta)-\langle x,\theta\rangle=h_{\Omega-x}(\theta)$ and for any compact convex set $K$ \cite[Equation (14)]{LangharstRoysdonZvavitch2022}
	$$m(K^\ast)=\frac{1}{n}\int_{S^{n-1}}h_K(\theta)^{-n}d\sigma(\theta),$$ allowing $\epsilon\to 0$ finishes the proof.
\end{proof}
Combining the above lemmas we can prove Theorem \ref{helson1}.
\\	\textbf{Proof of Theorem \ref{helson1}.}  Let, without loss of generality (Lemma \ref{unbounded}), $\Omega$ to be bounded. For a convex set centrally symmetric (see \cite{Kuperberg2008}) it holds that 
$$m(K)m(K^\ast)\geq c_n\frac{\pi^n}{n!},$$ where $$c_n=\frac{2^n (n!)^3}{(2n)!\Gamma(n/2+1)^2}.$$ Letting $K=(\Omega-x)\cap (x-\Omega)$ we get that
$$\frac{1}{\omega_\Omega(x)}=\frac{1}{m(K)}\leq \frac{n!}{c_n \pi^n}m(K^\ast).$$ 
We can observe that 
$$K^\ast=\conv ((\Omega-x)^\ast,-(\Omega-x)^\ast),$$ and thus by a Rogers-Shephard type inequality (see \cite{AlonsoGutierrezEtAl2021,RogersShephard1958}),
$$m(K^\ast)= m(\conv ((\Omega-x)^\ast,-(\Omega-x)^\ast))\leq 2^nm((\Omega-x)^\ast).$$
Combining the above inequalities with Lemma \ref{psilemma} we get that
$$\frac{1}{\omega_\Omega(x)}\leq \frac{2^nn!}{c_n\pi^n}\widehat{\Psi_\Omega}(x).$$
For $f\in \PW^1(\Omega)$,
\begin{eqnarray} \int_\Omega\dfrac{|\hat{f}(x)|^2}{\omega_\Omega(x)}dx &\leq& \frac{2^nn!}{c_n\pi^n} \int_\Omega|\hat{f}(x)|^2\widehat{\Psi_\Omega}(x)dx =  \frac{2^nn!}{c_n\pi^n} \langle f\ast \Psi_\Omega,f\rangle \nonumber \\
	&\leq&\frac{2^nn!}{c_n\pi^n} \|\Psi_\Omega\|_{L^\infty}\|f\|_{L^1}^2=\frac{4^n}{c_n}\|f\|_{L^1}^2. \nonumber 
\end{eqnarray}
The proof is complete. \qed

Let us notice that $c_n$ is an increasing sequence with $\frac{4}{\pi}\leq c_n\leq \sqrt{2}$.

The symmetric Mahler conjecture (see \cite{Mahler1939,FradeliziMeyerZvavitch2023}) suggests that the stronger inequality 
$$m(K)m(K^\ast)\geq \frac{4^n}{n!},$$ should hold. This would imply the stronger inequality  $$\int_{\Omega}\dfrac{|\hat{f}(x)|^2}{\omega_\Omega(x)}dx\leq \pi^n\|f\|_{L^1}^2, \quad f\in\PW^1(\Omega).$$ 
This conjecture holds for $n=1,2,3$ (see \cite{Mahler1939Minimal,IriyehShibata2020}) and thus the above inequality also holds.

\section{Hankel Operators and Nehari's Theorem}\label{sec3}
In this section we are going to translate Theorem \ref{helson1} into the language of Hankel operators on Paley-Wiener spaces, giving a positive answer to a question posed in \cite{MR4502777} for any convex set that does not contain affine lines. The question was originally posed only for the disc, and very recently was answered for the ball in $\mathbb{R}^n$.

Let $\Omega\subset\mathbb{R}^n$ be a convex set that does not contain affine lines. For a distribution $\widehat{\phi}$, we define the Hankel operator generated by the symbol $\phi$, $\Ha_\phi:\PW(\Omega)\to \PW(\Omega)$ by 
$$\widehat{\Ha_\phi f}(x)=\int_{\Omega}\widehat{\phi}(x+y)\hat{f}(y)dy,\quad x\in \Omega.$$
These operators are well defined on functions in $\mathcal{F}^{-1}C_c^\infty(\Omega)$ which are dense in $\PW(\Omega)$ \cite[Lemma 2.6]{bampouras2026boundaryweightedfourierinequalitiesconvex}. A simple observation shows that $\|\Ha_\phi\|\leq \|\phi\|_{L^\infty}$, and thus every bounded function generates a bounded Hankel operator. Nehari's theorem \cite{nehari1957bounded} states that for the case of $\Omega=\mathbb{R}_+$, every bounded Hankel operator is generated by some bounded function. For a bounded non-polygonal set this question was given a negative answer \cite{MR4700194,MR4502777}. Instead of asking whether every bounded Hankel operator is generated by a bounded symbol, Brevig and Perfekt in \cite{MR4502777} asked whether, in the case $\Omega=\mathbb{D}$, every Hilbert-Schmidt Hankel operator is generated by a bounded symbol. The Hilbert-Schmidt norm of an integral operator is the $L^2$ norm of its kernel, thus
$$\|\Ha_\phi\|_{S^2}^2=\int_\Omega \int_\Omega |\widehat{\phi}(x+y)|^2dxdy=2^{-n}\int_{\mathbb{R}^n}|\widehat{\phi}(x)|^2\omega_{2\Omega}(x)dx.$$

It turns out that this question is equivalent to Helson's inequality.
\begin{theo}\label{hilbertnehari}
	Let $\Omega\subset\mathbb{R}^n$ be a convex set that does not contain affine lines. Then every Hilbert-Schmidt Hankel operator on $\PW(\Omega)$ is generated by a bounded symbol.
\end{theo}
\begin{proof}
	Let $f\in \mathcal{F}^{-1}C_c^\infty(2\Omega).$ Then, by Theorem \ref{helson1} and the Cauchy-Schwarz inequality,
	$$|\langle f,\phi\rangle|=|\int_{2\Omega}\widehat{\phi}(x)\sqrt{\omega_{2\Omega}(x)}\frac{\hat{f}(x)}{\sqrt{\omega_{2\Omega}(x)}}dx|\leq C_n\|\Ha_\phi\|_{S^2}\|f\|_{L^1}.$$ Therefore, if $\Ha_\phi\in S^2$, the functional $Tf=\langle f,\phi\rangle$ is a bounded functional on $\PW^1(2\Omega)$, and thus by Hahn-Banach and the duality of $L^1$, there exists a bounded function $\psi$ such that $\langle f,\phi\rangle=\langle f,\psi\rangle$ for all $f\in \PW^1(2\Omega).$ Equivalently $\widehat{\phi}=\widehat{\psi}$ on $2\Omega$ as desired.
\end{proof}

In \cite[Section 7.2]{bampouras2026boundaryweightedfourierinequalitiesconvex}, the Hilbert matrix was introduced to provide a new proof of the failure of the Nehari theorem for the ball. The Hilbert matrix is defined as the operator $\mathcal{H}_\Omega:L^2(\Omega)\to L^2(\Omega)$ by
$$\mathcal{H}_\Omega(f)(x)=\int_\Omega \dfrac{f(y)}{\omega_{2\Omega}(x+y)}dy\quad x\in \Omega.$$
We start by observing that this operator is always bounded.
\begin{prop}\label{Hilbert}
	Let $\Omega\subset\mathbb{R}^n$ be a convex set that does not contain affine lines. Then 
	$$\|\mathcal{H}_\Omega\|\leq\frac{2^n(2n)!\Gamma(n/2+1)^2}{(n!)^3}.$$
\end{prop}
\begin{proof}
	By the proof of Theorem \ref{helson1}, if $\Omega_r=\Omega\cap B(0,r)$ we have that
	$$\frac{1}{\omega_{2\Omega_r}(x)}\leq \frac{2^n n!}{c_n\pi^n}\widehat{\Psi_{2\Omega_r}}(x), \quad x\in 2\Omega_r.$$
	Therefore, for $f,g\in L^2(\Omega_r)$,
	\begin{eqnarray} 
		|\langle \mathcal{H}_\Omega f,g\rangle|& \leq&\frac{2^n n!}{c_n\pi^n} \int_{\Omega_r}\int_{\Omega_r} |f(x)||g(y)|\widehat{\Psi_{2\Omega_r}}(x+y)dxdy \nonumber \\ &\leq& \frac{2^n n!}{c_n\pi^n}\|\Psi_{2\Omega_r}\|_{L^\infty}\|f\|_{L^2}\|g\|_{L^2}.\nonumber 
	\end{eqnarray}
	The proof is complete since $\bigcup_{r>0} L^2(\Omega_r)$ is dense in $L^2(\Omega)$.
\end{proof}

We will finish this section by observing that \cite[Lemma 7.5]{bampouras2026boundaryweightedfourierinequalitiesconvex} combined with Proposition \ref{Hilbert} shows that Nehari's theorem implies Hardy's inequality, i.e. the boundary-weighted inequality with $(p,q,d)=(1,1,1)$.
\begin{lm}
	Nehari's theorem implies Hardy's inequality, i.e.
	$$\int_\Omega\dfrac{|\hat{f}(x)|}{\omega_\Omega(x)}dx\leq C(\Omega)\|f\|_{L^1},\quad f\in \PW^1(\Omega).$$
\end{lm}

\section{Schatten class Hankel Operators}\label{Schatten}
In this section we are going to completely characterize Schatten class Hankel operators for Paley-Wiener spaces of a polyhedron. In order to present the characterization, we need some definitions.
\subsection*{Admissible decomposition.} Let $\Omega\subset\mathbb{R}^n$ be a convex set that does not contain affine lines. For $a>1$, an $a$-admissible decomposition of $\Omega$ is a collection of parallelepipeds $\{A_j^i\}_{(j,i)\in I}$, $I\subset \mathbb{Z}^2$ that satisfy the axioms
\begin{enumerate}
	\item There exists $M>0$ with $A_j^i\subset \omega_\Omega^{-1}(a^{j-M},a^{j+M}).$
	\item There exist $c_1,c_2>0$ such that $c_1a^j\leq m(A_j^i)\leq c_2 a^j.$
	\item There exists $\epsilon\in (0,1)$ such that
	$$\Omega= \bigcup_{(j,i)\in I}\epsilon A_j^i,$$ where $\epsilon A_j^i$ is the dilation of $A_j^i$ by $\epsilon$ with respect to its center.
	\item There exists $C>0$ such that whenever $A_j^i\cap A_k^l\neq \emptyset$, it holds that 
	$$T_{k,l}A_j^i\subset C(-1,1)^n,$$ where $T_{k,l}$ is the affine bijection that maps $A_k^l$ onto $(-1,1)^n$.
	\item There exists $N>0$ such that for any choice $(j,i),(k,l)\in I$, $$\#\{(\beta,\gamma):A_\beta^\gamma \cap \frac{1}{2}(A_j^i+A_k^l)\neq\emptyset\}\leq N.$$
\end{enumerate}

The existence of such a decomposition of $\Omega$ ensures that we can find a family of smooth functions $\phi_j^i\in \mathcal{F}^{-1}C_c^\infty(\mathbb{R}^n),$ satisfying 
\begin{enumerate}
	\item $\supp\widehat{\phi}_j^i\subset A_j^i$,
	\item $\sup_{(j,i)\in I}\|\phi_j^i\|_{L^1}<\infty,$ and
	\item $$\sum_{(j,i)\in I}\widehat{\phi}_j^i(x)=\chi_\Omega(x).$$
\end{enumerate}
The collection of such families of smooth functions will be denoted by $\mathscr{P}(A_j^i)$. \cite[Proposition 2.6]{bampouras2024besovspacesschattenclass} and \cite[Theorem 1.9]{bampouras2026boundaryweightedfourierinequalitiesconvex} ensure that the above always exist.
\begin{theo}
	Let $\Omega\subset\mathbb{R}^n$ be a convex set that does not contain affine lines and let $a>1$. Then $\Omega$ is $a$-admissible and for any admissible decomposition $A_j^i$, $\mathscr{P}(A_j^i)\neq \emptyset.$
\end{theo}
We can use this decomposition to define appropriate Besov spaces.
\subsection*{Besov spaces.} 
For $a>1$ and an $a$-admissible decomposition $A_j^i$ of a convex set $\Omega\subset\mathbb{R}^n$ that does not contain affine lines, we define the Besov space $B_{p,q}^s(\Omega)$, $1\leq p,q<\infty$, $s\in\mathbb{R}$ by the completion of
$$\{f\in \mathcal{F}^{-1}C_c^\infty(\Omega):(a^{js}\|f\ast\phi_j^i\|_{L^p})_{(j,i)\in I}\in\ell^q(j,i)\}$$ over the metric
$$\|f\|_{B_{p,q}^s(\Omega)}=\left(\sum_{j,i}a^{jsq}\|f\ast\phi_j^i\|_{L^p}^q\right)^{1/q}.$$
These Besov spaces have been used to characterize Schatten class Hankel operators, (see \cite{bampouras2024besovspacesschattenclass}), for example, for any $1\leq p\leq 2$ and any convex set $\Omega\subset\mathbb{R}^n$ not containing affine lines it holds that
$$\|\Ha_\phi\|_{S^p}\approx \|\phi\|_{B_{p,p}^{1/p}(2\Omega)}.$$
To extend this result to $p>2$ we will need to interpolate with $p=\infty$. To achieve that we need to introduce extended Hankel operators $\Ha_\phi^{\sigma,\tau}:\PW(\Omega)\to \PW(\Omega)$, $\sigma,\tau\in\mathbb{R}$ by
$$\widehat{\Ha_\phi^{\sigma,\tau}f}(x)=\int_{\Omega}\widehat{\phi}(x+y)\hat{f}(y)\omega_{\Omega}^\sigma(x)\omega_\Omega^\tau(y)dy,\quad x\in\Omega.$$

To get an endpoint result -- in order to be able to interpolate -- we only need to characterize the boundedness of the integral operator $K_{\sigma,\tau}:L^2(\Omega)\to L^2(\Omega)$ with kernel
$$K_{\sigma,\tau}(x,y)=\frac{\omega_\Omega^\sigma(x)\omega_{\Omega}^\tau(y)}{\omega_\Omega^{\sigma+\tau+1}(\frac{x+y}{2})}.$$ Indeed, \cite[Corollary 1.5, Lemma 3.9 and Theorem 3.10]{bampouras2024besovspacesschattenclass} together with interpolation yield the following.
\begin{prop}\label{polyhconseq}
	Let $P\subset\mathbb{R}^n$ be a polyhedron that does not contain affine lines and assume that for every $\sigma,\tau>-1/2$ the operator $K_{\sigma,\tau}$ is bounded. Then for $1\leq p<\infty$ and every
	$$\sigma,\tau>\max(-\frac{1}{2},-\frac{1}{p}), \quad \sigma+\tau>-\frac{1}{p},$$ it holds that
	$$\Ha_\varphi^{\sigma,\tau}\in S^p(\PW(P))
	\quad\Longleftrightarrow\quad
	\varphi\in B_{p,p}^{\sigma+\tau+1/p}(2P),$$
	with equivalence of norms.
\end{prop}
Whenever we refer to a polyhedron $P\subset\mathbb{R}^n$ with no affine lines, we will also assume that $P$ can be written as
$$P=\{x\in\mathbb{R}^n:\ell_j(x)>0, 1\leq j\leq N\}, \quad \ell_j(x)=b_j-\langle a_j, x\rangle.$$
We begin by establishing an estimate for the $\omega_P$ function.
\begin{lm}\label{omegaformula}
	Let $P\subset\mathbb{R}^n$ be a polyhedron with no affine lines.
	Let $\mathcal{B}$ be the collection of all sets $I\subset \{1,...,N\}$ for which $\# I =n$ and the matrix $A_I$ with rows $a_i, i\in I$ is invertible, and set
	$$q_I(x)=\dfrac{\prod_{i\in I}\ell_i(x)}{|\det A_I|}.$$
	Then $$\left(\frac{2}{n}\right)^n \min_{I\in \mathcal{B}}q_I(x)\leq \omega_P(x)\leq 2^n \min_{I\in \mathcal{B}}q_I(x).$$
\end{lm}
\begin{proof}
	For $x\in P$, let $v_j=a_j/\ell_j(x).$ A simple observation shows that
	$$M_P(x)-x=\{u\in\mathbb{R}^n:|\langle v_j,u\rangle|<1, 1\leq j\leq N\},$$ where $M_P(x)=P\cap (2x-P)$. Thus,
	$$M_P(x)-x\subset \{u\in\mathbb{R}^n:|\langle v_i,u\rangle|<1, i\in I\},$$ for every $I\in\mathcal{B}$. Since $ \{u\in\mathbb{R}^n:|\langle v_i,u\rangle|<1, i\in I\}$ has measure $2^n/|\det(v_i)_{i\in I}|$ we get that
	$$\omega_P(x)=m(M_P(x)-x)\leq \min_{I\in\mathcal{B}}2^n\frac{1}{|\det(v_i)_{i\in I}|}=2^n\min_{I\in\mathcal{B}}q_I(x).$$ Now let us fix $I'$ such that $$|\det(v_i)_{i\in I'}|=\max_{I\in \mathcal{B}} |\det(v_i)_{i\in I}|,$$ and let us write $v_j=\sum_{i\in I'}c_{j,i}v_i$, $1\leq j\leq N$. Cramer's rule and maximality give $|c_{j,i}|\leq 1$ and thus
	$$\{u\in\mathbb{R}^n:|\langle v_i,u\rangle|<1/n, i\in I'\}\subset \{u\in\mathbb{R}^n:|\langle v_j,u\rangle|<1, 1\leq j \leq N\}.$$ Taking measures gives us the lower bound.
\end{proof}
Let us now introduce some notation.
$$d\mu(x)=\frac{dx}{\omega_P(x)},\quad \quad R(x,y)=\frac{\sqrt{\omega_P(x)\omega_P(y)}}{\omega_P((x+y)/2)}$$
and the logarithmic metric
$$d_P(x,y)=\max_{1\leq j\leq N}\left|\log\frac{\ell_j(x)}{\ell_j(y)}\right|.$$

These quantities satisfy the following properties.
\begin{lm}\label{metricestimates}
	Let $P\subset\mathbb{R}^n$ be a polyhedron with no affine lines. Then, for any $y\in P,$ $r>0$,
	$$\mu(\{x\in P:d_{P}(x,y)<r\})\leq \binom{N}{n}n^n r^n.$$
	Moreover, 
	$$R(x,y)\leq 2^n\sqrt{n} e^{-d_P(x,y)/2}, \quad x,y\in P.$$
\end{lm}
\begin{proof}
	We will only consider the cases $n\geq 2$, the case $n=1$ is immediate.
	For $I\in\mathcal{B}$, let 
	$$E_I=\{x\in P:q_I(x)=\min_{J\in\mathcal{B}}q_J(x)\}.$$ 
	By Lemma \ref{omegaformula} and using the change of variables $s_i=\log \ell_i(x)$, $i\in I$, which gives $$ds=\frac{dx}{q_I(x)},$$ we get
	\begin{eqnarray} 
		\mu(\{x\in P:d_{P}(x,y)<r\})&\leq& \sum_{I\in\mathcal{B}} \mu(\{x\in E_I:d_{P}(x,y)<r\}) \nonumber \\
		&=& \sum_{I\in\mathcal{B}} \int_{\{x\in E_I:d_{P}(x,y)<r\}} \frac{1}{\omega_P(x)}dx \nonumber \\
		&\leq& (n/2)^n\sum_{I\in\mathcal{B}} \int_{\{x\in E_I:d_{P}(x,y)<r\}} \frac{dx}{q_I(x)} \nonumber\\
		&\leq& (n/2)^n\sum_{I\in\mathcal{B}} \int_{\{|s_i-\log \ell_i(y)|<r\; :\; i\in I\}} ds \nonumber \\
		&\leq& (n/2)^n\sum_{I\in\mathcal{B}} (2r)^n \leq\binom{N}{n}n^nr^n. \nonumber
	\end{eqnarray} 
	Now let us fix $x,y\in P$, and set $$t=\max_{1\leq i\leq N}\frac{|\langle a_i,(y-x)/2\rangle|}{\ell_i((x+y)/2)}=\frac{|\langle a_j ,(y-x)/2\rangle|}{\ell_j((x+y)/2)}<1.$$ Let us also choose $e\in \{(x-y)/2,(y-x)/2\}$ so that
	$$\ell_j((x+y)/2+e)=\ell_j((x+y)/2)(1-t).$$
	In addition, 
	\begin{align*} M_P((x+y)/2+e)-((x+y)/2+e)&\subset  \\ 2 (M_P((x+y)/2)-(x+y)/2)\cap & \{u\in\mathbb{R}^n:|\langle a_j,u\rangle|<\ell_j((x+y)/2)(1-t)\}.
	\end{align*}
	Brunn's concavity principle (see \cite[Theorem~8.4]{Gruber2007}) says that for a convex body $K$ and $\theta\in S^{n-1}$, the function $$A(t)=m_{n-1}^{1/(n-1)}(K\cap\{x\in\mathbb{R}^{n}:\langle x,\theta\rangle=t\})$$
	is concave on $\{t\in\mathbb{R}:A(t)>0\}$. Therefore, for a linear functional $Lx=\langle u,x\rangle$ we derive that
	$$A_L(t)=m_{n-1}^{1/(n-1)}(K\cap \{x\in\mathbb{R}^{n}:Lx=t\})$$
	is concave on $\{t:A_L(t)>0\}$. If, in addition, $K$ is centrally symmetric and $\sup_{x\in K}|Lx|=1$, then
	$$m(K)=\frac{1}{|u|}\int_{-1}^{1}A_L^{n-1}(t)dt\geq \frac{2}{n|u|}A_L^{n-1}(0),$$
	and therefore
	\begin{eqnarray}
		m(2K\cap \{x\in\mathbb{R}^{n}:|Lx|<s\})
		&=&\frac{2^n}{|u|}\int_{-s/2}^{s/2}A_L^{n-1}(t)dt\nonumber \\
		&\leq& \frac{2^ns}{|u|}A_L^{n-1}(0)\nonumber \\
		&\leq& n2^{n-1}s m(K). \nonumber 
	\end{eqnarray}
	Applying this to the above estimate with
	$$K=M_P((x+y)/2)-(x+y)/2,\quad \text{and} \quad L(u)=\left\langle u,\frac{a_j}{\ell_j((x+y)/2)}\right\rangle,$$
	we have $\sup_{u\in K}|L(u)|=1$, and hence
	$$\omega_P((x+y)/2+e)\leq n2^{n-1}(1-t)\omega_P((x+y)/2).$$
	Concavity of the $\omega_P^{1/n}$ function also gives that $$\omega_P((x+y)/2-e)\leq 2^n \omega_P((x+y)/2).$$ The proof is complete by observing that $d_P(x,y)=\log\frac{1+t}{1-t}$.
\end{proof}
Using this we can derive the boundedness of the integral operator with kernel $$K_{\sigma,\tau}(x,y)=\frac{\omega_P^\sigma(x)\omega_P^\tau(y)}{\omega_P^{\sigma+\tau+1}((x+y)/2)}.$$

\begin{theo}\label{Kpolyh}
	Let $P\subset\mathbb{R}^n$ be a polyhedron that does not contain affine lines. The integral operator defined on $L^2(P)$ with kernel $K_{\sigma,\tau}$ is bounded whenever $\sigma,\tau>-1/2.$
\end{theo}
\begin{proof}
	By concavity of $\omega_P^{1/n}$, $\omega_P((x+y)/2)\geq 2^{-n}\max(\omega_P(x),\omega_P(y))$. For $\rho=\min(\sigma,\tau)$ we can write
	\begin{eqnarray} 
		\int_{P}K_{\sigma,\tau}(x,y)\left(\frac{\omega_P(x)}{\omega_P(y)}\right)^{-1/2}dx&=&\int_{P}K_{\sigma,\tau}(x,y)(\omega_P(x)\omega_P(y))^{1/2}d\mu(x) \nonumber \\
		&\leq& 2^{n|\sigma-\tau|}\int_{P}R^{2\rho+1}(x,y)d\mu(x).\nonumber
	\end{eqnarray}
	By Lemma \ref{metricestimates}, we get that
	\begin{align*} 
		\int_{P}K_{\sigma,\tau}(x,y)&\left(\frac{\omega_P(x)}{\omega_P(y)}\right)^{-1/2}dx\leq 2^{n|\sigma-\tau|}(2^n\sqrt{n})^{2\rho+1}\int_{P}e^{-d_P(x,y)(\rho+1/2)} d\mu(x)  \\
		&\leq 2^{n|\sigma-\tau|}(2^n\sqrt{n})^{2\rho+1} \sum_{k=0}^{\infty}e^{-k(\rho+1/2)}\mu(\{x\in P:d_P(x,y)<k+1\})  \\
		&\leq  2^{n|\sigma-\tau|}(2^n\sqrt{n})^{2\rho+1}\binom{N}{n}n^n\sum_{k=0}^{\infty}e^{-k(\rho+1/2)} (k+1)^n ,
	\end{align*} 
	which is uniformly bounded in $y$ since $\rho+1/2>0.$ The result follows by Schur's test.
\end{proof}

Therefore, by Proposition \ref{polyhconseq} and Theorem \ref{Kpolyh} we have the following corollary.

\begin{cor} Let $1\leq p<\infty$ and suppose that $$\sigma,\tau>\max(-\frac{1}{2},-\frac{1}{p}), \quad \sigma+\tau>-\frac{1}{p}.$$
	Then, for every polyhedron $P\subset \mathbb{R}^n$ that does not contain affine lines, it holds that
	$$\Ha_\varphi^{\sigma,\tau}\in S^p(\PW(P))
	\quad\Longleftrightarrow\quad
	\varphi\in B_{p,p}^{\sigma+\tau+1/p}(2P),$$
	with equivalence of norms.
\end{cor}

\section{Weak type factorizations}\label{sec5}
In this section we are going to relate Nehari's consequence for Schatten class Hankel operators to a weak factorization of $\PW^1$. It is well known that Nehari's theorem is equivalent to the following statement:
\\ There exists $C>0$ such that for every $f\in \PW^1(2\Omega)$ we can find $f_j,g_j\in \PW(\Omega)$ with $f=\sum_{j} f_jg_j$ (the convergence is discussed at length below) and $$\sum_j \|f_j\|_{L^2}\|g_j\|_{L^2}\leq C\|f\|_{L^1}.$$ We are going to show that Nehari's theorem for the Hankel operators in $S^p$ is related to the weaker factorization
$$f=\sum_{j} f_jg_j,\quad \sum_j (\|f_j\|_{L^2}\|g_j\|_{L^2})^{p'}\leq C_p\|f\|_{L^1}^{p'}.$$

Let $\Omega$ be a convex set that does not contain affine lines. For $1\leq p\leq \infty$ we define the set $W_p(\Omega)$ consisting of all functions $f$ that can be written as $f=\sum_{j=1}^N f_j g_j$, where $N\in\mathbb{Z}_+$, $f_j,g_j,$ $j=1,...,N$ are orthogonal families in $\PW(\Omega)\setminus\{0\}$ and we endow it with the norm $$\|f\|_{W_p(\Omega)}=\inf  \left\| (\|f_j\|_{L^2}\|g_j\|_{L^2})_{1\leq j\leq N}\right\|_{\ell^p},$$ where the infimum is taken over all representations of $f$. First, we establish that $\|\cdot\|_{W_p(\Omega)}$ is a norm.

\begin{lm}\label{WpNorm}
	For $1\leq p\leq\infty$, $\|\cdot\|_{W_p(\Omega)}$ is a norm on $W_p(\Omega)$.
\end{lm}
\begin{proof}
	The property $\|\lambda f\|_{W_p(\Omega)}=|\lambda|\| f\|_{W_p(\Omega)}$ is evident. Let us take $f\in W_p(\Omega)$ with $\|f\|_{W_p(\Omega)}=0$. For a representation $f=\sum_{j=1}^N f_j g_j$, let us define the conjugate-linear operator $T:\PW(\Omega)\to \PW(\Omega)$ as
	$$Tx=\sum_{j=1}^N\langle f_j,x\rangle g_j.$$
	Since $$ \left\| (\|f_j\|_{L^2}\|g_j\|_{L^2})_{1\leq j\leq N}\right\|_{\ell^p}=\|T\|_{S^p},$$ we can find representations $f_j^k,g_j^k,$ $j=1,...,N_k$ and conjugate linear operators $T_k$ such that $$T_kx=\sum_{j=1}^{N_k}\langle f_j^k,x\rangle g_j^k,\quad \|T_k\|_{S^p}\to 0 \quad \text{ and } f=\sum_{j=1}^{N_k}f_j^k g_j^k.$$
	Then, for any Schwartz function $\phi$ with Fourier transform compactly supported in $2\Omega$, by H\"{o}lder's inequality,
	\begin{eqnarray} 
		|\langle \phi,f\rangle|&=&|\sum_{j=1}^{N_k}\langle \phi, f_j^k g_j^k\rangle|\leq\sum_{j=1}^{N_k}\|f_j^k\|_{L^2}\|g_j^k\|_{L^2}|\langle \Ha_\phi (\frac{f_j^k}{\|f_j^k\|_{L^2}})^\ast, \frac{g_j^k}{\|g_j^k\|_{L^2}}\rangle| \nonumber \\ &\leq& \|T_k\|_{S^p}\|\Ha_\phi\|_{S^{p'}}\to 0, \nonumber
	\end{eqnarray} 
	where $g^\ast$ denotes the function defined by the property $\widehat{g^\ast}=\overline{\hat{g}}.$
	Since $\supp\hat{f}\subset 2\Omega$, $f=0$. Now let us fix $f,g\in W_p(\Omega)$. For $\epsilon>0$ let $f_j,f'_j$, $j=1,...,N_1$ and $g_j,g_j'$, $j=1,...,N_2$ be representations of $f$ and $g$, respectively, such that
	$$\|f\|_{W_p(\Omega)}\geq \left\| (\|f_j\|_{L^2}\|f'_j\|_{L^2})_{1\leq j\leq N_1}\right\|_{\ell^p}-\epsilon,$$ and  $$\|g\|_{W_p(\Omega)}\geq \left\| (\|g_j\|_{L^2}\|g'_j\|_{L^2})_{1\leq j\leq N_2}\right\|_{\ell^p}-\epsilon.$$
Let us also define as before the conjugate linear operators
	$$Tx=\sum_{j=1}^{N_1}\langle f_j,x\rangle f_j',\quad Sx=\sum_{j=1}^{N_2}\langle g_j,x\rangle g_j'.$$ Then the operator $T+S$ is a finite rank conjugate linear operator, and thus we can find orthogonal families $h_j,h_j'\in \PW(\Omega)$, $j=1,...,N_3$ with 
	$$(T+S)x=\sum_{j=1}^{N_3}\langle h_j,x\rangle h_j'.$$ Now, if $e_1,e_2,...$ is an orthonormal basis for $\PW(\Omega)$, we can notice that
	\begin{eqnarray} h&=&\sum_{j=1}^{N_3}h_jh_j'=\sum_{j\geq 1}(T+S)(e_j)e_j=\sum_{j\geq 1}T(e_j)e_j+\sum_{j\geq 1}S(e_j)e_j \nonumber \\
		&=&\sum_{j=1}^{N_1}f_jf_j'+\sum_{j=1}^{N_2}g_jg_j'=f+g,\nonumber
	\end{eqnarray}
	and that
	$$\|h\|_{W_p(\Omega)}\leq \|T+S\|_{S^p}\leq \|T\|_{S^p}+\|S\|_{S^p}\leq \|f\|_{W_p(\Omega)}+\|g\|_{W_p(\Omega)}+2\epsilon.$$
	Taking $\epsilon\to 0$ we get the triangle inequality.
\end{proof}
For a distribution $f$, we say that $f\in (W_p(\Omega))'$ if for any $\phi\in W_p(\Omega)$, 
$$|\langle f,\phi\rangle|\leq C\|\phi\|_{W_p(\Omega)},$$ and we write
$$\|f\|_{(W_p(\Omega))'}=\sup_{\phi\in W_p(\Omega)}\dfrac{|\langle f,\phi\rangle|}{\|\phi\|_{W_p(\Omega)}}.$$ It is evident that this space coincides with the Banach dual of $W_p(\Omega)$; indeed, every $T$ in the dual space induces the bounded Hankel form

$$
(f,g)\longmapsto T(fg),
$$

and hence, by the Schwartz kernel theorem, there exists a distribution $\phi_T$ such that

$$
T(fg)=\langle \phi_T,fg\rangle,\qquad f,g\in\PW(\Omega).
$$

It turns out that this space is isometrically isomorphic to $S^{p'}$ Hankel operators.
\begin{lm}\label{dualSp}
	For $1\leq p\leq \infty$ it holds that
	$$\|\phi\|_{(W_{p}(\Omega))'}=\|\Ha_\phi\|_{S^{p'}}.$$
\end{lm}
\begin{proof}
	First, let us observe that for $f\in W_{p}(\Omega)$, with a representation $f=\sum_{j=1}^N f_jg_j$,
	\begin{eqnarray}
		|\langle \phi, f\rangle|&\leq& \sum_{j=1}^N|\langle \phi, f_j g_j\rangle|=\sum_{j=1}^N\|f_j\|_{L^2}\|g_j\|_{L^2}|\langle \Ha_\phi \frac{f_j^{\ast}}{\|f_j\|_{L^2}},  \frac{g_j}{\|g_j\|_{L^2}}\rangle| \nonumber \\
		&\leq& \left\| (\|f_j\|_{L^2}\|g_j\|_{L^2})_{1\leq j\leq N}\right\|_{\ell^{p}}\|\Ha_\phi\|_{S^{p'}}.\nonumber
	\end{eqnarray}
	Taking the infimum over all representations of $f$ we get
	$$|\langle \phi, f\rangle|\leq  \|f\|_{W_{p}(\Omega)}\|\Ha_\phi\|_{S^{p'}},$$  and thus
	$$\|\phi\|_{(W_{p}(\Omega))'}\leq \|\Ha_\phi\|_{S^{p'}}.$$
	For the converse inequality, we observe that for any $f_j,g_j$, $j=1,...,N$, orthonormal families,
	\begin{align*}
		\| (\langle \Ha_\phi f_j,g_j\rangle)_{1\leq j\leq N}\|&_{\ell^{p'}}
		= \sup_{\|c_j\|_{\ell^{p}}= 1}|\sum_jc_j\langle \Ha_\phi f_j,g_j\rangle| =   \sup_{\|c_j\|_{\ell^{p}}= 1}|\langle \phi, \sum_jc_j f_j^\ast g_j\rangle| \nonumber \\
		&\leq \|\phi\|_{(W_{p}(\Omega))'} \sup_{\|c_j\|_{\ell^{p}}= 1}\|\sum_jc_j f_j^\ast g_j\|_{W_{p}(\Omega)} \leq   \|\phi\|_{(W_{p}(\Omega))'} .\nonumber
	\end{align*}
	Since
	$$\|\Ha_\phi\|_{S^{p'}}=\sup_{f_j,g_j\text{ finite orthonormal families}}	\|(\langle \Ha_\phi f_j,g_j\rangle)_{j}\|_{\ell^{p'}},$$
	taking the supremum over all such $f_j,g_j$ families finishes the proof.
\end{proof}
In order to relate Nehari's theorem to a factorization of $\PW^1(\Omega)$ we need the following lemma.

\begin{lm}\label{sumsdense}
	The set
	$$\left\{ \sum_{j=1}^{N}f_jg_j: N\in\mathbb{Z}_+, j=1,...,N\text{ and } f_j,g_j\in \PW(\Omega) \text{ orthogonal families} \right\}$$
	is dense in $\PW^1(2\Omega)$.
\end{lm}
\begin{proof}
	By \cite[Lemma 2.6]{bampouras2026boundaryweightedfourierinequalitiesconvex}, it suffices to show that these functions are dense in the set of smooth functions with compactly supported Fourier transform in $2\Omega$. Carlsson and Perfekt \cite[proof of Proposition 5.1]{MR4227573} showed that for any $f\in \mathcal{F}^{-1}C_c^\infty(2\Omega)$ we can find $f_j,g_j\in \PW(\Omega)$ such that
	$$f=\sum_{j=1}^{\infty} f_jg_j\text{ in $L^1$},\quad \sum_j \|f_j\|_{L^2}\|g_j\|_{L^2}<\infty.$$
	It suffices to show that this sum can be replaced by orthogonal sequences. Let us define the conjugate linear operator $T:\PW(\Omega)\to \PW(\Omega)$ by
	$$Tx=\sum_{j=1}^\infty\langle f_j,x\rangle g_j.$$ Since 
	$$\|T\|_{S^1}\leq \sum_j \|f_j\|_{L^2}\|g_j\|_{L^2}<\infty,$$
	this operator belongs to the trace class; thus, we can find orthogonal sequences $f_j',g_j'$, such that
	$$Tx=\sum_{j=1}^{\infty} \langle f'_j,x\rangle g'_j.$$ Now, if $e_1,e_2,...$ is an orthonormal basis of $\PW(\Omega)$, it can be computed that
	$$\sum_{j=1}^\infty T(e_j)e_j=\sum_{j=1}^\infty f_jg_j =\sum_{j=1}^\infty f'_jg'_j.$$
	The proof is complete.
\end{proof}

We will also need the following characterization of the completion of $W_p(\Omega)$.

\begin{prop}\label{completion}
	Let $\Omega$ be a convex set that does not contain affine lines and let $\mathcal{W}_p(\Omega)$, $1\leq p\leq \infty$, be the completion of $W_p(\Omega)$. Then, if $p<\infty$, $f\in \mathcal{W}_p(\Omega)$ if and only if there are orthogonal sequences $f_j,g_j\in \PW(\Omega)$ such that
	$$f=\lim_{N\to \infty} \sum_{j=1}^N f_j g_j \text{ in $\mathcal{W}_p(\Omega)$,} \quad \| (\|f_j\|_{L^2}\|g_j\|_{L^2})_{j}\|_{\ell^p}<\infty.$$
	Furthermore
	$$\|f\|_{\mathcal{W}_p(\Omega)}=\inf \|(\|f_j\|_{L^2}\|g_j\|_{L^2})_j\|_{\ell^p},$$ where the infimum is taken over all representations of $f$.
	For $p=\infty$ the proposition holds if we replace the condition $\| (\|f_j\|_{L^2}\|g_j\|_{L^2})_{j}\|_{\ell^\infty}<\infty$ by $\|f_j\|_{L^2}\|g_j\|_{L^2}\to 0$.
\end{prop}
\begin{proof}
	We will only treat the case $1\leq p<\infty.$ The case $p=\infty$ follows similarly, replacing $S^\infty$ by compact operators. For $f\in W_p(\Omega)$, let us define for a representation of $f=\sum_j f_jg_j$ the conjugate linear operator $T:\PW(\Omega)\to \PW(\Omega)$ as
	$$Tx=\sum_j \langle f_j,x\rangle g_j.$$
Let also $e_1,e_2,...$ be an orthonormal basis for $\PW(\Omega)$. Then, we can notice that
	$$\|f\|_{W_p(\Omega)}=\inf\{\|T\|_{S^p}: T\text{ is finite rank conjugate linear with }\sum_j T(e_j)e_j=f\}.$$
	It is therefore evident that
	$$\|f\|_{\mathcal{W}_p(\Omega)}=\inf\{\|T\|_{S^p}: T\in S^p\text{ conjugate linear with }\sum_j T(e_j)e_j=f \text{ in }\mathcal{W}_p(\Omega)\}.$$
	Since any such operator $T\in S^p$ can be written as 
	$$Tx=\sum_j \langle f_j,x\rangle g_j,$$ for $f_j,g_j\in \PW(\Omega)\setminus\{0\}$ orthogonal sequences, we have that
	$$\|T\|_{S^p}=\left(\sum_j (\|f_j\|_{L^2}\|g_j\|_{L^2})^p\right)^{1/p}.$$
	The proof is complete.
\end{proof}
Now we can prove that Nehari's theorem is equivalent to a factorization of $\PW^1(2\Omega)$.
\begin{prop}\label{factor}
	Let $1\leq p\leq \infty$ and let $\Omega$ be a convex set that does not contain affine lines. Then the following are equivalent.
	\begin{enumerate}
		\item Nehari's theorem holds for the Schatten class $S^p$.
		\item There exists $C>0$ such that, for all $f\in \PW^1(2\Omega)$, the inequality
		$$\|f\|_{\mathcal{W}_{p'}(\Omega)}\leq C\|f\|_{L^1},$$ holds.
	\end{enumerate}
\end{prop}
\begin{proof}
	As observed in the proof of Theorem \ref{hilbertnehari} and \cite[Lemma 2.1]{bampouras2026nehari}, Nehari's theorem for $S^p$ is equivalent to the estimate
	$$|\langle f,\phi\rangle|\lesssim \|\Ha_\phi\|_{S^p}\|f\|_{L^1},\quad f\in\PW^1(2\Omega).$$
	Let us assume that Nehari's analogue holds for Schatten class Hankel operators in $S^p$. Then, by Lemma \ref{dualSp}, for $f\in W_{p'}(\Omega)$
	$$\|f\|_{W_{p'}(\Omega)}=\sup_{\phi\in (W_{p'}(\Omega))'}\frac{|\langle f,\phi\rangle|}{\|\phi\|_{(W_{p'}(\Omega))'}}=\sup_{\phi\in (W_{p'}(\Omega))'}\frac{|\langle f,\phi\rangle|}{\|\Ha_\phi\|_{S^p}}\lesssim \|f\|_{L^1}.$$
	Since the elements of $W_{p'}(\Omega)$ are dense in $\PW^1(2\Omega)$ (Lemma \ref{sumsdense}), we get that
	$$\|f\|_{\mathcal{W}_{p'}(\Omega)}\lesssim  \|f\|_{L^1}.$$
	For the converse, again by Lemma \ref{dualSp}, for $f\in W_{p'}(\Omega)$,
	$$|\langle f,\phi\rangle|\leq \|f\|_{W_{p'}(\Omega)}\|\phi\|_{(W_{p'}(\Omega))'}\lesssim \|f\|_{L^1}\|\Ha_\phi\|_{S^p}.$$ Again by Lemma \ref{sumsdense}, we get that for any $f\in \PW^1(2\Omega)$,
	$$|\langle f,\phi\rangle|\lesssim \|f\|_{L^1}\|\Ha_\phi\|_{S^p}.$$
	The proof is complete.
\end{proof}
Let us notice that, by the proof of Lemma \ref{sumsdense}, for $W_1(\Omega)$ we can ignore the orthogonality condition. Therefore, Nehari's theorem holds if and only if any $f\in \PW^1(2\Omega)$ can be written as $f=\sum_{j=1}^\infty f_j g_j$, where $f_j,g_j$ are any sequences (not necessarily orthogonal) in $\PW(\Omega)$ with $\sum_{j=1}^\infty \|f_j\|_{L^2}\|g_j\|_{L^2}<\infty.$

We finish this paper by observing that Theorem \ref{hilbertnehari} and Proposition \ref{factor} give the following consequence.
\begin{cor}
	Let $\Omega\subset\mathbb{R}^n$ be a convex set that does not contain affine lines. There is a constant $C>0$ such that for any $f\in \PW^1(2\Omega)$, there are $f_j,g_j$ orthogonal sequences in $\PW(\Omega)$ with
	$$f=\lim_{N\to \infty} \sum_{j=1}^N f_j g_j \text{ in $\mathcal{W}_2(\Omega)$,} \quad  \text{and}\quad  \sum_j (\|f_j\|_{L^2}\|g_j\|_{L^2})^2\leq C\|f\|_{L^1}^2.$$
\end{cor}

	\bibliographystyle{plain}
	\bibliography{Helson}
	
\end{document}